\documentclass{article}
\usepackage[numbers,square]{natbib}
\usepackage{graphicx}
\usepackage{amsmath}
\usepackage{amsthm}
\usepackage{amssymb}%

\newcommand{\R}{\mathbb{R}}
\newcommand{\N}{\mathbb{N}}

\newcommand{\Z}{\mathbb{Z}}

\renewcommand{\P}{\mathbb{P}}

\newcommand{\essinf}{\mathop{\mathrm{essinf}}\nolimits}

\newcommand{\eps}{\varepsilon}

\theoremstyle{plain}
\newtheorem{theorem}{Theorem}[section]
\newtheorem{lemma}{Lemma}[section]

\newtheorem{proposition}{Proposition}[section]

\theoremstyle{definition}

\theoremstyle{remark}
\newtheorem{remark}{Remark}[section]

\begin{document}

\title{Ergodic Properties of Max--Infinitely Divisible Processes}
\author{Zakhar Kabluchko\footnote{Corresponding author} \footnote{Institut f\"ur Mathematische Stochastik, Georg-August-Universit\"at G\"ottingen, Goldschmidtstra\ss e 7, D-37077, G\"ottingen, Germany. Email: kabluch@math.uni-goettingen.de}, Martin Schlather\footnote{Institut f\"ur Mathematische Stochastik, Georg-August-Universit\"at G\"ottingen, Goldschmidtstra\ss e 7, D-37077, G\"ottingen, Germany. Email: schlather@math.uni-goettingen.de}}


\maketitle

\begin{abstract}
We prove that a stationary max--infinitely divisible process is mixing (ergodic) iff its dependence function converges  to $0$ (is Cesaro summable to $0$). These criteria are applied to some classes of max--infinitely divisible processes.
\end{abstract}

\section{Introduction and statement of results}\label{sec:intro}
Given a class of stationary stochastic processes, an important question is how to characterize ergodicity and mixing for members of this class. A classical result of this type is due to \citet{maruyama49}, who showed that a stationary Gaussian process $X$ is mixing iff its covariance function $r$ satisfies $\lim_{t\to\infty}r(t)=0$, and ergodic if its spectral measure has no atoms (see e.g.~\cite{cornfeld_etal_book}). Later, \citet{maruyama70} studied the more general class of stationary infinitely--divisible (i.d.)\ processes and gave an easy verifiable criterion for mixing in terms of the two--dimensional distributions of the process.
The results of~\cite{maruyama70} were used by a number of authors~\cite{cambanis_etal87,gross_robertson93,gross94,podgorski92}
to obtain various characterizations of ergodicity and mixing for stationary $\alpha$--stable processes.
A final simplification of these conditions was achieved by \citet{rosinski_zak96,rosinski_zak97} who obtained ergodicity and mixing criteria for stationary i.d.\ processes in terms of the so--called codifference function, a natural coefficient measuring the dependence within an i.d.\ process. A criterion for ergodicity in terms of the spectral representation of the process was established in~\cite{samorodnitsky05}, see also~\cite{roy07}.

The aim of the present paper is to give simple necessary and sufficient conditions for ergodicity and mixing of stationary \textit{max--infinitely divisible} (max--i.d.)\ processes. A stochastic process $X=\{X(t), t\in \Z\}$ is called max--i.d.\ if for every $n\in \N$ it can be represented as a maximum, taken componentwise, of $n$ independent identically distributed stochastic processes. Equivalently, $X$ is max--i.d.\ if all its finite--dimensional distributions belong to the class of multivariate max--i.d.\ distributions. Max--i.d.\ distributions (and processes) were introduced  in~\cite{balkema_resnick77} and studied in~\cite{vatan85,resnick_book,gine_etal90}.

A major role in our characterization will be played by a dependence coefficient $\tau_a(t)$  which is defined as follows.
Let a stationary max--i.d.\ process $\{X(t), t\in\Z\}$ be given. Define $l=\essinf X(0)\in[-\infty,\infty)$.  For $t\in\Z$ and $a>l$ we define the coefficient $\tau_a(t)$ measuring the dependence between $X(0)$ and $X(t)$ by
\begin{equation}\label{eq:def_tau}
\tau_{a}(t)=\log \P[X(0)\leq a, X(t)\leq a]-\log \P[X(0)\leq a]-\log \P[X(t)\leq a].
\end{equation}
Note that by stationarity of $X$, $\tau_a(t)=\tau_a(-t)$.

The next two theorems are our main results. We refer to~\cite{cornfeld_etal_book} for the basic notions from the ergodic theory.

\begin{theorem}\label{theo:mix}
Let $\{X(t), t\in\Z\}$ be a stationary max--i.d.\ process with dependence function $\tau_a(t)$.
Then the following conditions are equivalent:
\begin{enumerate}
\item \label{p:mix1} $X$ is mixing.
\item \label{p:mix1a} $X$ is mixing of all orders.
\item \label{p:mix2} For every $a>l$, $\lim_{t\to\infty}\tau_a(t)=0$.
\end{enumerate}
\end{theorem}
\begin{theorem}\label{theo:erg}
Let $\{X(t), t\in\Z\}$ be a stationary max--i.d.\ process with dependence function $\tau_a(t)$.
Then the following conditions are equivalent:
\begin{enumerate}
\item \label{p:erg1} $X$ ergodic.
\item \label{p:erg2} $X$ is weakly mixing.
\item \label{p:erg3} For every $a>l$, $\lim_{n\to\infty}\frac{1}{n}\sum_{t=1}^n\tau_a(t)=0$.
\end{enumerate}
\end{theorem}
For max--stable processes (which form a subclass of the class of max--i.d.\ processes), ergodicity and mixing were studied in~\cite{weintraub91,stoev07,kabluchko_extremes}. We will see in Section~\ref{sec:max_stab} below that our Theorem~\ref{theo:mix} characterizing mixing is a generalization of a max--stable result of~\cite[Theorem~3.3]{stoev07}, whereas Theorem~\ref{theo:erg} characterizing ergodicity is new even in the max--stable case.

\begin{remark}
For simplicity, we restrict ourselves to stochastic processes indexed by $\Z$. All results hold for stochastically continuous processes on $\R$, the proofs being the same.
\end{remark}

The rest of the paper is organized as follows.
Theorem~\ref{theo:mix} and Theorem~\ref{theo:erg} will be proved in Section~\ref{sec:proof}. In Section~\ref{sec:examples} we apply our results to some particular classes of max--i.d.\ processes.

\section{Proof of Theorem~\ref{theo:mix} and Theorem~\ref{theo:erg}}\label{sec:proof}
\subsection{Notation}\label{sec:facts_proof}
We start by recalling some facts about multivariate max--i.d.\ laws, see~\cite{resnick_book} or~\cite{vatan85} for more information. A $d$--variate random vector $X=(X_1,\ldots,X_d)$ is called max--i.d.\ if for every $n\in\N$ it can be represented as a maximum (taken componentwise) of $n$ independent identically distributed random vectors.
Let $F(y)=\P[X \leq y]$, where $y=(y_1,\ldots,y_d)\in\R^d$, be the distribution function of a max--i.d.\ random vector $X$. Then there is a measure $Q$ on $\bar \R^d=[-\infty,\infty)^d$, called the exponent measure of $X$, such that
\begin{equation}\label{eq:exp_meas}
F(y)=e^{-Q([-\infty,y]^c)} \;\;\; \forall y>\essinf X,
\end{equation}
where $[-\infty,y]=[-\infty, y_1]\times\ldots\times[-\infty,y_d]$,  the superscript $c$ denotes the complement in $\bar \R^d$, and the $\essinf$ is taken componentwise.
\begin{remark}\label{rem:marginals}
The exponent measure of some marginal of $X$, say $(X_1,\ldots, X_m)$, where $m\leq d$, is the projection of the  measure $Q$ onto the first $m$ coordinates (see e.g.~\cite{vatan85}).
\end{remark}

Let $\{X(t),t\in\Z\}$ be a stationary max--i.d.\ process. Given $t_1,\ldots, t_k\in\Z$, we denote by $Q_{t_1,\ldots,t_k}$ the exponent measure of the max--i.d.\ random vector $(X(t_1),\ldots, X(t_k))$. Recall that $l=\essinf X_t$.
\begin{lemma}\label{lem:tau_exp_meas}
For all $a>l$ and $t\in\Z$, $\tau_a(t)=Q_{0,t}((a,+\infty)\times (a,+\infty))$.
\end{lemma}
\begin{proof}
By~\eqref{eq:exp_meas}, we have
$$
\P[X(0)\leq a, X(t)\leq a]=e^{-Q_{0,t}(\bar \R^2\backslash[-\infty,a]^2)}.
$$
Again by~\eqref{eq:exp_meas} and by Remark~\ref{rem:marginals},
$$
\P[X(0)\leq a]=e^{-Q_0((a,\infty))}=e^{-Q_{0,t}((a,\infty)\times\bar \R)}.
$$
Similarly, $\P[X(t)\leq a]=e^{-Q_{0,t}(\bar \R\times (a,\infty))}$.
Inserting this into~\eqref{eq:def_tau}, we obtain the statement of the lemma.
\end{proof}

Fix some $k\in\N$, some integers $t_1',\ldots,t_k'$ and $t_1'',\ldots,t_k''$, and some  $y'=(y_1',\ldots,y_k')\in(l,\infty)^d$, $y''=(y_1'',\ldots,y_k'')\in(l,\infty)^d$.  Define the random events $A$ and $B$ by
\begin{align}
A&=\{X(t_i')\leq y_i', i=1,\ldots,k\},\label{eq:def_A} \\
B&=\{X(t_i'')\leq y_i'', i=1,\ldots,k\}. \label{eq:def_B}
\end{align}
Further, for $t\in\Z$, we define
\begin{equation}\label{eq:def_Bt}
B_t=\{X(t_i''+t)\leq y_i'', i=1,\ldots,k\}.
\end{equation}

\subsection{Basic lemma}
Our proofs will be based on the following lemma.
\begin{lemma}\label{lem:ineq}
Let $a$ be the smallest of the numbers $y_i',y_j''$ ($i,j=1,\ldots,k$). Then for every $t\in\Z$, we have
\begin{equation}\label{eq:lem}
\P[A]\P[B]\leq \P[A\cap B_t]\leq \exp\left(\sum_{i=1}^k \sum_{j=1}^k \tau_{a}(t+t_j''-t_i')\right) \P[A]\P[B].
\end{equation}
\end{lemma}
\begin{remark}
The above lemma is similar to an inequality due to \citet{lebowitz72} in the form stated in~\cite[Theorem~2]{newman82}. The basic fact which lies behind the proof of the lemma is that the process $X$, being max--i.d., is positively associated. For the latter result, see~\cite{resnick_book}.
\end{remark}
\begin{proof}[Proof of Lemma~\ref{lem:ineq}]
To shorten the notation, we define the measures $\Lambda'$ and $\Lambda''$ on $\bar \R^k$ by
$$
\Lambda'=Q_{t_1',\ldots, t_k'},\;\;\; \Lambda''=Q_{t_1'',\ldots, t_k''}.
$$
Let also $\Lambda_t$ be a measure on $\bar \R^{2k}$ defined by
$$
\Lambda_t=Q_{t_1',\ldots, t_k', t_1''+t,\ldots, t_k''+t}.
$$
Write $y=(y_1',\ldots,y_k',y_1'',\ldots,y_k'')$.
Define the sets $E', E'' \subset \bar \R^k$ and $E\subset \bar \R^{2k}$ by
$$
E'=[-\infty, y']^c,\;\;\; E''=[-\infty, y'']^c,\;\;\; E=[-\infty, y]^c.
$$
Note that $E=(E'\times \bar \R^k )\cup (\bar \R^k\times E'')$. Therefore,
$$
\Lambda_t(E) \leq \Lambda_t(E'\times \bar \R^k)+\Lambda_t(\bar \R^k \times E'')=\Lambda'(E')+\Lambda''(E'').
$$
Note that the last equality follows from Remark~\ref{rem:marginals}. By~\eqref{eq:exp_meas}, we have $\P[A\cap B_t]=e^{-\Lambda_t(E)}$, $\P[A]=e^{-\Lambda'(E')}$, and $\P[B]=e^{-\Lambda''(E'')}$. It follows that for every $t\in\Z$,
$$
\P[A\cap B_t]=e^{-\Lambda_t(E)}\geq e^{-(\Lambda'(E')+\Lambda''(E''))}=\P[A]\P[B],
$$
which proves the first inequality in~\eqref{eq:lem}.

On the other hand, we have
\begin{equation}\label{eq:ineq_bonfer}
\Lambda_t(E)=\Lambda_t(E'\times \bar \R^k)+\Lambda_t(\bar \R^k\times E'')-\Lambda_t(F)=\Lambda'(E')+\Lambda''(E'')-\Lambda_t(F),
\end{equation}
where $F=(E'\times \bar \R^k)\cap (\bar \R^k\times E'')$. For $i,j\in\{1,\ldots,k\}$ let
$$
F_{ij}=\{(z_1',\ldots, z_k', z_1'',\ldots, z_k'')\in\bar\R^{2k}: z_i'>y_i', z_j''>y_j''\}.
$$
Then $F= \cup_{i=1}^k\cup_{j=1}^k F_{ij}$ and $\Lambda_t(F_{ij})=Q_{t_i',t_j''+t}((y_i',+\infty)\times (y_j'',+\infty))$ (again by Remark~\ref{rem:marginals}).
Hence,
$$
\Lambda_t(F)
\leq \sum_{i=1}^k\sum_{j=1}^k \Lambda_t(F_{ij})=
\sum_{i=1}^k\sum_{j=1}^k Q_{t_i',t_j''+t}((y_i',+\infty)\times (y_j'',+\infty)).
$$
Recall that $a\leq y_i'$ and $a\leq y_j''$. Using stationarity of $X$ and Lemma~\ref{lem:tau_exp_meas}, we may write
$$
\Lambda_t(F)\leq \sum_{i=1}^k\sum_{j=1}^k Q_{t_i',t_j''+t}((a,+\infty)\times (a,+\infty))= \sum_{i=1}^k \sum_{j=1}^k \tau_{a}(t+t_j''-t_i').
$$
Inserting this into~\eqref{eq:ineq_bonfer} and recalling that $\P[A]=e^{-\Lambda'(E')}$ and $\P[B]=e^{-\Lambda''(E'')}$, we obtain
\begin{align*}
\P[A\cap B_t]
&=e^{-\Lambda_t(E)}\\
&= e^{-\Lambda'(E')-\Lambda''(E'')+\Lambda_t(F)}\\
&\leq e^{\sum_{i=1}^k \sum_{j=1}^k \tau_{a}(t+t_j''-t_i')} \P[A]\P[B].
\end{align*}
This completes the proof of the second inequality in~\eqref{eq:lem}.
\end{proof}

\subsection{Proofs}
\begin{proof}[Proof of Theorem~\ref{theo:mix}.]
Recall that the process $X$ is mixing iff for every random events $A$, $B$ as in~\eqref{eq:def_A}, \eqref{eq:def_B},
\begin{equation}\label{eq:cond_mix}
\lim_{t\to\infty} \P[A\cap B_t]=\P[A]\P[B].
\end{equation}
Suppose first that $X$ is mixing. By~\eqref{eq:cond_mix} with $A=B=\{X(0)\leq a\}$, where $a>l$, this implies that
$$
\lim_{t\to\infty} \P[X(0)\leq a, X(t)\leq a]=\P[X(0)\leq a]\P[X(t)\leq a].
$$
Taking the logarithm and recalling~\eqref{eq:def_tau}, we obtain $\lim_{t\to\infty}\tau_a(t)=0$. This proves the implication \ref{p:mix1}$\Rightarrow$\ref{p:mix2} of the theorem.

To prove the implication \ref{p:mix2}$\Rightarrow$\ref{p:mix1}, suppose that $\lim_{t\to\infty}\tau_a(t)=0$ for every $a>l$. Let $A$, $B$ be arbitrary as in~\eqref{eq:def_A}, \eqref{eq:def_B}. Note that
$$
\lim_{t\to \infty}\sum_{i=1}^k \sum_{j=1}^k \tau_{a}(t+t_j''-t_i')=0.
$$
Applying Lemma~\ref{lem:ineq}, we see that~\eqref{eq:cond_mix} holds. This shows that $X$ is mixing and completes the proof of the implication \ref{p:mix2}$\Rightarrow$\ref{p:mix1}.

The implication \ref{p:mix1a}$\Rightarrow$\ref{p:mix1} holds trivially. Thus, to complete the proof, we need to prove the implication \ref{p:mix2}$\Rightarrow$\ref{p:mix1a}. Let $A^{(0)}, A^{(1)},\ldots$ be random events of the form
$$
A^{(j)}=\{X(t^{(j)}_i)\leq y_i^{(j)}, i=1,\ldots,k\}
$$
for some $k\in\N$, $t_i^{(j)}\in\Z$, $y_i^{(j)}>l$ ($i=1,\ldots,k$; $j=0,1,\ldots$). For $t\in\Z$ define
$$
A^{(j)}_t=\{X(t^{(j)}_i+t)\leq y_i^{(j)}, i=1,\ldots,k\}.
$$
Recall that the process $X$ is mixing of order $r\in\N$ iff for every $A^{(0)},\ldots, A^{(r)}$ as above, we have
\begin{equation}\label{eq:cond_mix_r}
\lim_{t_1,\ldots,t_r\to+\infty} \P[A^{(0)}\cap A^{(1)}_{t_1} \cap \ldots A^{(r)}_{t_1+\ldots+t_r}]=\P[A^{(0)}] \ldots \P[A^{(r)}].
\end{equation}
Suppose now that Condition~\ref{p:mix2} holds, i.e.\ $\lim_{t\to\infty}\tau_a(t)=0$ for all $a>l$. We have already shown that this implies~\eqref{eq:cond_mix_r} with $r=1$. To handle the general case, we use an induction on $r$.  Assume that we have proved~\eqref{eq:cond_mix_r} for some $r=p\in\N$. Define
$$
A=A(t_1,\ldots,t_p):=A^{(0)}\cap A^{(1)}_{t_1} \cap \ldots A^{(p)}_{t_1+\ldots+t_p}
$$
and $B=A^{(p+1)}$. By Lemma~\ref{lem:ineq},
\begin{equation}\label{eq:188}
\P[A]\P[B]
\leq \P[A\cap B_{t_1+\ldots+t_{p+1}}]
\leq e^{\theta}\P[A]\P[B],
\end{equation}
where $\theta=\theta_{t_1,\ldots,t_{p+1}}$ is defined by
$$
\theta_{t_1,\ldots,t_{p+1}}=\sum_{i'=1}^k\sum_{i''=1}^k \sum_{j=1}^{p} \tau_a((t_1+\ldots+t_{p+1})+t_{i''}^{(p+1)}-(t_1+\ldots+t_j)-t_{i'}^{(j)}).
$$
By Condition~\ref{p:mix2}, we have $\lim_{t_{p+1}\to+\infty}\theta_{t_1,\ldots,t_{p+1}}=0$. Therefore, by~\eqref{eq:188},
$$
\lim_{t_1,\ldots,t_{p+1}\to+\infty} \P[A\cap B_{t_1+\ldots+t_{p+1}}]=\lim_{t_1,\ldots,t_p\to+\infty}\P[A]\P[B]=\P[A^{(0)}] \ldots \P[A^{(p+1)}],
$$
where the second equality follows from the fact that by the inductive assumption, $\lim_{t_1,\ldots,t_p\to+\infty}\P[A]=\P[A^{(0)}] \ldots \P[A^{(p)}]$.
It follows that
$$
\lim_{t_1,\ldots,t_{p+1}\to+\infty} \P[A^{(0)}\cap A^{(1)}_{t_1} \cap \ldots A^{(p+1)}_{t_1+\ldots+t_{p+1}}]=\P[A^{(0)}] \ldots \P[A^{(p+1)}].
$$
This proves~\eqref{eq:cond_mix_r} with $r=p+1$ and completes the proof of the implication \ref{p:mix2}$\Rightarrow$\ref{p:mix1a}.
\end{proof}
To prove Theorem~\ref{theo:erg}, we will need the following elementary lemma.
\begin{lemma}\label{lem:cesaro}
Let $\{\theta_t\}_{t\in\N}$ be a sequence such that  $0\leq \theta_t\leq C$ for some constant $C$ and all $t\in\N$. Then the following statements are equivalent:
\begin{enumerate}
\item For some (equivalently, every) $\kappa>0$, $\lim_{n\to\infty}\frac 1n \sum_{t=1}^{n} e^{\kappa\theta_t}=1$.
\item $\lim_{n\to\infty} \frac 1n \sum_{t=1}^{n} \theta_t=0$.
\end{enumerate}
\end{lemma}
\begin{proof}
Since the exponential function $\theta\mapsto e^{\kappa \theta}$ is convex, we have
$$
1+\kappa\theta_t\leq e^{\kappa\theta_t}\leq 1+ C^{-1}(e^{\kappa C}-1)\theta_t.
$$
The statement of the lemma follows readily.
\end{proof}
\begin{proof}[Proof of Theorem~\ref{theo:erg}]
Recall that the process $X$ is ergodic iff for every events $A$, $B$ as in~\eqref{eq:def_A}, \eqref{eq:def_B},
\begin{equation}\label{eq:cond_erg}
\lim_{n\to\infty}\frac 1n \sum_{t=1}^n \P[A\cap B_t]=\P[A]\P[B].
\end{equation}
Further, $X$ is weakly mixing iff for every $A$, $B$ as in~\eqref{eq:def_A}, \eqref{eq:def_B},
\begin{equation}\label{eq:cond_weak_mix}
\lim_{n\to\infty}\frac 1n \sum_{t=1}^n |\P[A\cap B_t]-\P[A]\P[B]|=0.
\end{equation}
By the first inequality in~\eqref{eq:lem}, conditions~\eqref{eq:cond_erg} and~\eqref{eq:cond_weak_mix} are clearly equivalent. This proves the equivalence \ref{p:erg1}$\Leftrightarrow$\ref{p:erg2} of the theorem.

Suppose now that $X$ is ergodic, which means that~\eqref{eq:cond_erg} holds. Taking $A=B=\{X(0)\leq a\}$ for some $a>l$, we may rewrite~\eqref{eq:cond_erg} in the form
\begin{equation}\label{eq:lim_frac_sum_exp}
\lim_{n\to\infty}\frac 1n \sum_{t=1}^n e^{\tau_a(t)}=1.
\end{equation}
Note that by Lemma~\ref{lem:tau_exp_meas},
\begin{equation}\label{eq:tau_is_bounded}
\tau_a(t)=Q_{0,t}((a,\infty)\times (a,\infty))\leq Q_{0,t}((a,\infty)\times \bar \R)=Q_0((a,\infty)),
\end{equation}
where the last  equality follows from Remark~\ref{rem:marginals}.
This implies that we have $0\leq \tau_a(t)\leq C$ for some constant $C$ and all $t\in\Z$.
An application of Lemma~\ref{lem:cesaro} to~\eqref{eq:lim_frac_sum_exp} yields that $\lim_{n\to\infty}\frac 1n \sum_{t=1}^n \tau_a(t)=0$. This proves the implication \ref{p:erg1}$\Rightarrow$\ref{p:erg3} of the theorem.

Now suppose that Condition~\ref{p:erg3} of the theorem holds, that is
\begin{equation}\label{eq:lim_sum_cesaro_tau}
\lim_{n\to\infty}\frac 1n \sum_{t=1}^n \tau_a(t)=0,\;\;\; \forall a>l.
\end{equation}
Let $A$, $B$ be random events as in~\eqref{eq:def_A}, \eqref{eq:def_B}. By Lemma~\ref{lem:ineq},
we have
\begin{equation}\label{eq:PA_PB_estim}
\P[A]\P[B] \leq \frac 1n  \sum_{t=1}^n\P[A\cap B_t] \leq \P[A]\P[B]\cdot \left (\frac 1n \sum_{t=1}^{n} e^{\theta_t}\right),
\end{equation}
where
$$
\theta_t=\sum_{i=1}^k \sum_{j=1}^k \tau_{a}(t+t_j''-t_i').
$$
By~\eqref{eq:tau_is_bounded}, there is $C$ such that $0\leq \theta_t\leq C$ for all $t\in\Z$.  It follows from~\eqref{eq:lim_sum_cesaro_tau} that $\lim_{n\to\infty} \frac 1n\sum_{t=1}^n \theta_t=0$.
Applying Lemma~\ref{lem:cesaro}, we obtain
$$
\lim_{n\to\infty}\frac 1n \sum_{t=1}^{n} e^{\theta_t}=1.
$$
Together with~\eqref{eq:PA_PB_estim}, this implies that~\eqref{eq:cond_erg} holds. Hence, $X$ is ergodic. This completes the proof of the implication \ref{p:erg3}$\Rightarrow$\ref{p:erg1} and the proof of the theorem.
\end{proof}

\section{Examples}\label{sec:examples}
\subsection{Ergodicity and mixing of max--stable processes}\label{sec:max_stab}
Recall that a stochastic process $X$ is called max--stable if for every $n\in\N$, the maximum of $n$ independent copies of $X$, taken componentwise, has the same law as $X$ up to an affine transformation. In the sequel, let $\{X(t), t\in\Z\}$ be a stationary max--stable process, and suppose that its marginals are $1$--Fr\'echet with unit scale parameter, i.e.\
\begin{equation}\label{eq:stand_marginals}
\P[X(t)\leq y]=e^{-1/y},\;\;\; t\in\Z, y>0.
\end{equation}
There is no loss of generality in assuming~\eqref{eq:stand_marginals}, since stationary max--stable processes with more general marginals can be reduced to the above class by simple transformations.

Our goal in this section is to give criteria for ergodicity and mixing of stationary max--stable processes. Since max--stable processes are max--i.d., Theorem~\ref{theo:mix} and Theorem~\ref{theo:erg} are applicable. We will see that in the max--stable case, the dependence coefficient $\tau_a(t)$ reduces to a natural dependence coefficient $r(t)$ which goes back to~\cite{sibuya60,oliveira62,dehaan85}, and is defined as follows. Let $t\in \Z$. By max--stability of the random vector $(X(0), X(t))$, there is $\varrho(t)>0$ such that
\begin{equation}\label{eq:def_varrho}
\P[X(0)\leq y, X(t)\leq y]=e^{-\varrho(t)/y},\;\;\; y>0.
\end{equation}
Then $r(t)$, a coefficient measuring the dependence between $X(0)$ and $X(t)$, is defined as  $r(t)=2-\varrho(t)$.
We have $r(t)\in[0,1]$, and the cases $r(t)=0$ and $r(t)=1$ correspond to independence and a.s.\ equality, respectively.
\begin{lemma}\label{lem:r_tau}
Let $\{X(t), t\in\Z\}$ be a stationary max--stable process satisfying~\eqref{eq:stand_marginals}. Then for every $a>0$, $\tau_a(t)=r(t)/a$.
\end{lemma}
\begin{proof}
It follows from~\eqref{eq:def_tau} combined with~\eqref{eq:stand_marginals}, \eqref{eq:def_varrho} that $\tau_a(t)=(2-\varrho(t))/a=r(t)/a$.
\end{proof}
As an immediate corollary of Theorem~\ref{theo:mix} and Lemma~\ref{lem:r_tau}, we obtain the following result of~\citet{stoev07}.
\begin{theorem}\label{theo:crit_mix_stab}
Let $\{X(t), t\in\Z\}$ be a stationary max--stable process such that~\eqref{eq:stand_marginals} is satisfied. Then $X$ is mixing iff $\lim_{t\to\infty}r(t)=0$.
\end{theorem}

Before stating our next result, recall that the (upper) asymptotic density of a set $D\subset\N$ is defined as $\limsup_{n\to\infty}\lambda(D_n)/n$, where $\lambda(D_n)$ is the number of elements in the set $D_n=D\cap \{1,\ldots,n\}$.  Note also that the sequence $r(t)$ is positive--definite (see~\cite{schlather_tawn03} or~\cite{stoev07}) and hence, by Bochner's theorem, there is a symmetric with respect to $0$ measure $\mu$ on $[-\pi,\pi]$  (called the spectral measure of $X$) such that $r(t)=\int_{-\pi}^{\pi}e^{itx}d\mu(x)$ for all $t\in\Z$.
\begin{theorem}\label{theo:crit_erg_stab}
Let $\{X(t), t\in\Z\}$ be a stationary max--stable process such that~\eqref{eq:stand_marginals} is satisfied. Then the following conditions are equivalent:
\begin{enumerate}
\item \label{p:1}  $X$ is ergodic.
\item \label{p:2}  $X$ is weakly mixing.
\item \label{p:3} $\lim_{n\to\infty}\frac{1}{n} \sum_{t=1}^n r(t)=0$.
\item \label{p:4} There is a set $D\subset \N$ of asymptotic density $0$ such that $\lim_{t\to+\infty, t\notin D} r(t)=0$.
\item \label{p:4a} For every $\eps>0$ there is a set $D_{\eps}\subset \N$ of asymptotic  density at most $\eps$ such that $\lim_{t\to+\infty, t\notin D_{\eps}} r(t)=0$.
\item \label{p:5} The spectral measure of $X$  has no atom at $0$.
\end{enumerate}
\end{theorem}
\begin{proof}
Using Theorem~\ref{theo:erg} and taking into account Lemma~\ref{lem:r_tau}, we see that Conditions~\ref{p:1}, \ref{p:2}, \ref{p:3} are equivalent. Recall that $r(t)\in[0,1]$. It is well known  that this implies that Conditions~\ref{p:3}, \ref{p:4}, \ref{p:4a} are equivalent, see e.g.~\cite[Lemma~1]{rosinski_zak97}. The fact that Condition~\ref{p:3} and Condition~\ref{p:5} are equivalent can be proved as follows. Note that for $t\in[-\pi,\pi]$,
$$
\lim_{n\to\infty}\frac 1n\sum_{t=1}^n e^{itx}=1_{\{0\}}(x),
$$
Note also that $\left|\frac 1n\sum_{t=1}^n e^{itx}\right|\leq 1$. By the bounded convergence theorem, we have, as $n\to\infty$,
$$
\frac{1}{n} \sum_{t=1}^n r(t)=\int_{-\pi}^{\pi}\left( \frac 1n\sum_{t=1}^n e^{itx}\right)d\mu(x)\to \int_{-\pi}^{\pi}1_{\{0\}}(x)d\mu(x)=\mu(\{0\}).
$$
This completes the proof.
\end{proof}

A condition for ergodicity of max--stable processes which is difficult to verify was given in~\cite[Theorem~3.2]{stoev07}. Later, the result of~\cite{stoev07} was used in~\cite{kabluchko_extremes} to show that a max--stable process is ergodic iff the flow generating its spectral representation has no positive recurrent component.
For symmetric $\alpha$--stable processes, a measure of dependence similar to $r(t)$ is called codifference, see~\cite{samorodnitsky_taqqu_book}. Theorem~\ref{theo:crit_mix_stab} and  Theorem~\ref{theo:crit_erg_stab}  are  max--stable counterparts of the $\alpha$--stable results of~\cite{rosinski_zak96,rosinski_zak97}.  It is also interesting to compare Condition~\ref{p:5} of Theorem~\ref{theo:crit_erg_stab} to a classical result of~\cite{maruyama49} saying that a stationary Gaussian sequence is ergodic iff its spectral measure has no atoms (and not only no atom at $0$). Let us stress that in sharp contrast to the Gaussian case, the knowledge of $r(t)$ (or, equivalently, $\mu$) does not  determine the law of the max--stable process $X$ completely. Nevertheless, ergodicity and mixing of $X$ can be characterized in terms of $r$ only.

\subsection{Ergodicity and mixing of Brown--Resnick processes}
Let us mention an application of Theorem~\ref{theo:crit_erg_stab} to Brown--Resnick processes, a class of max--stable processes which was introduced in~\cite{kabluchko_schlather_dehaan07} and which is defined as follows. Let $W_i$, $i\in\N$, be independent copies of a zero--mean stochastically continuous Gaussian process $\{W(t), t\in\R\}$ with stationary increments, $W(0)=0$, and variance $\sigma^2(t)$. Independently, let $\{U_i,i\in\N\}$ be a Poisson point process on $\R$ with intensity $e^{-x}$. Then the process
\begin{equation}\label{eq:def_brown_resnick}
X(t)=\max_{i\in\N}e^{U_i+W_i(t)-\sigma^2(t)/2}
\end{equation}
is max--stable with marginals satisfying~\eqref{eq:stand_marginals}. It was shown in~\cite{kabluchko_schlather_dehaan07} that $X$ is stationary and that the dependence function of $X$ is given by
\begin{equation}\label{eq:varrho_eq_bar_phi}
r(t)=\bar \Phi (\sigma(t)/2),
\end{equation}
where $\bar \Phi(z)=(2\pi)^{-1/2} \int_{z}^{\infty} e^{-x^2/2}dx$ is the tail of the standard Gaussian distribution. By a continuous--time version of Theorem~\ref{theo:crit_mix_stab}, this implies that the process $X$ is mixing iff $\lim_{t\to\infty}\sigma^2(t)=\infty$. This fact was noted in~\cite{kabluchko_schlather_dehaan07}, whereas a characterization of ergodicity remained open. Now we are able to fill this gap.
\begin{proposition}\label{prop:crit_erg_br}
Let $X$ be as in~\eqref{eq:def_brown_resnick}. Then the following conditions are equivalent:
\begin{enumerate}
\item \label{p:br1} $X$ is ergodic.
\item \label{p:br2} There is a measurable set $D\subset \R_+$ of asymptotic density $0$ such that
$\lim_{t\to+\infty, t\notin D} \sigma^2(t)=\infty$.
\item \label{p:br3} For every $\eps>0$ there is a measurable set $D_{\eps}\subset \R_{+}$ of density at most $\eps$ such that $\lim_{t\to+\infty, t\notin D_{\eps}} \sigma^2(t)=\infty$.
\end{enumerate}
\end{proposition}
\begin{proof}
The statement follows from~\eqref{eq:varrho_eq_bar_phi} and a continuous--time version of Theorem~\ref{theo:crit_erg_stab}.
\end{proof}

We complete this section with an example which shows that Brown--Resnick processes can exhibit a rather exotic behavior.
\begin{proposition}\label{prop:counter_br}
There exists a Brown--Resnick process which is ergodic but non--mixing.
\end{proposition}
\begin{proof}
Define a function $\sigma^2(t)$ by
\begin{equation}\label{eq:def_sigma_counter}
\sigma^2(t)=\sum_{k=1}^{\infty}(1-\cos(2\pi t/2^{k})).
\end{equation}
The elementary inequality $1-\cos z\leq z^2/2$ shows that the series on the right-hand side of~\eqref{eq:def_sigma_counter} converges uniformly on compacts. Hence, the function $\sigma^2$ is finite and continuous. Further, the function $\sigma^2$ is negative definite (see e.g.~\cite{kabluchko_schlather_dehaan07}) since the function $t\mapsto 1-\cos(at)$ is negative definite for every $a$.  We claim that the Brown--Resnick process $X$ corresponding to $\sigma^2$ is ergodic but non--mixing.

First we show that $X$ is non--mixing. Take $t=2^n$ for some $n\in\N$. Then $1-\cos(2\pi t/2^k)=0$ for $k=1,\ldots,n$, so that the first $n$ summands on the right--hand side of~\eqref{eq:def_sigma_counter} vanish. If $k=n+l$ for some $l\in\N$, then the  inequality $1-\cos z\leq z^2/2$ implies that $1-\cos(2\pi t/2^k)\leq 2\pi^2 2^{-2l}$. Therefore,
$$
\sigma^2(2^n)\leq 2\pi^2 \sum_{l=1}^{\infty}2^{-2l}=2\pi^2/3.
$$
By a result of~\cite{stoev07,kabluchko_schlather_dehaan07} mentioned above, this implies that $X$ is non--mixing.

Now let us show that $X$ is ergodic. Take some small $\eps>0$. We will construct a set $D_{\eps}$ satisfying Condition~\ref{p:br3} of Proposition~\ref{prop:crit_erg_br}.
For every $k\in\N$ define a set $A_k=\cup_{i\in\Z} [(i-\eps) 2^k, (i+\eps) 2^k]$.
For $n,k\in\N$ let $B_{k,n}=[2^n,2^{n+1}]\backslash A_k$.
For $t\in[2^n,2^{n+1}]$ we define a function $S_n(t)=\sum_{k=1}^n 1_{B_{k,n}}(t)$.
Let $D_{n,\eps}=\{t\in[2^n,2^{n+1}]: S_n(t)<(1/2)n\}$. We claim that the set $D_{\eps}=\cup_{n\in\N}D_{n,\eps}$ satisfies the requirements of Condition~\ref{p:br3} of Proposition~\ref{prop:crit_erg_br}.

First let us estimate the asymptotic density of $D_{\eps}$. To this end, we show that
\begin{equation}\label{eq:lambda_D_n_eps}
m_{n,\eps}:=\lambda(D_{n,\eps})\leq 6\eps\cdot 2^n,
\end{equation}
where $\lambda$ is the Lebesgue measure.
An easy calculation shows that $\lambda(B_{k,n})=(1-2\eps)2^{n}$. It follows that
\begin{equation}\label{eq:int_Sn}
\int_{2^n}^{2^{n+1}}S_n(t)dt=(1-2\eps)2^n n.
\end{equation}
To prove~\eqref{eq:lambda_D_n_eps}, assume that $m_{n,\eps}>6\eps\cdot 2^n$. Estimating $S_n(t)$ by $(1/2)n$ if $t\in D_{n,\eps}$, and by $n$ otherwise, we obtain
$$
\int_{2^n}^{2^{n+1}}S_n(t)dt\leq (1/2)n m_{n,\eps}+(2^n-m_{n,\eps})n= 2^nn -(1/2)m_{n,\eps}n < 2^n n- 3\eps \cdot2^nn.
$$
However, this contradicts~\eqref{eq:int_Sn}. This completes the proof of~\eqref{eq:lambda_D_n_eps}.

Now, the upper density of $D_{\eps}$ can be estimated as follows. Let $t\in[2^n,2^{n+1}]$ for some $n\in\N$. Then we have
$$
\lambda([0,t]\cap D_{\eps})\leq \lambda([0, 2^{n+1}]\cap D_{\eps})=\sum_{k=1}^n \lambda(D_{k,\eps})\leq 6\eps \sum_{k=1}^n 2^k <12 \eps 2^{n}\leq 12\eps t.
$$
It follows that
$$
\limsup_{t\to+\infty}t^{-1}\lambda([0,t]\cap D_{\eps})\leq 12\eps,
$$
which proves that the asymptotic density of $D_{\eps}$ does not exceed $12\eps$.

It remains to prove that
\begin{equation}\label{eq:lim_t_notin_D}
\lim_{t\to+\infty, t\notin D_{\eps}}\sigma^2(t)=\infty.
\end{equation}
For every $t\notin A_k$, the distance from the number $t/2^k$ to $\Z$ is at least $\eps$. Therefore, we have $1-\cos(2\pi t/2^k)>c(\eps)$ for all $t\notin A_k$, where $c(\eps)>0$ is some constant not depending on $k$.
Now take some $t\in[2^n,2^{n+1}]$. If additionally $t\notin D_{n,\eps}$, then
$$
\sigma^2(t)\geq \sum_{k=1}^n (1-\cos(2\pi t/2^k))\geq c(\eps) \sum_{k=1}^n 1_{B_{k,n}}(t)=c(\eps) S_n(t)\geq (1/2)c(\eps)n,
$$
where the last inequality follows from the definition of the set $D_{n,\eps}$. It follows that there is a constant $C$ such that for every $t>2$, $t\notin D_{\eps}$, we have
$\sigma^2(t)\geq C\log t$.
This completes the proof of~\eqref{eq:lim_t_notin_D} and the proof of the proposition.
\end{proof}

\subsection{Distance to the nearest particle in an ideal gas}
In this section we apply Theorem~\ref{theo:mix} to show that a process studied by \citet{penrose88,penrose91,penrose92} is mixing. This process describes the distance from the origin to the nearest particle in a gas of independent Brownian particles, and is defined as follows. Consider an infinite number of particles starting at the points of a Poisson point process with unit intensity on $\R^d$ and performing independently of each other Brownian motions. In other words, the position of the $i$--th particle at time $t$ is given by $U_i+W_i(t)$, where $\{U_i, i\in\N\}$ is a homogeneous Poisson point process on $\R^d$, and $W_i$, $i\in\N$, are independent copies of an $\R^d$--valued Brownian motion $\{W(t), t\in\R\}$.  Denote by $\|\cdot\|_2$ the Euclidian norm on $\R^d$. Then we are interested in the process $\{X(t), t\in\R\}$ defined by
\begin{equation}\label{def:semi_min_stab}
X(t)=\min_{i\in\N} \|U_i+W_i(t)\|_2.
\end{equation}
By~\cite{penrose92}, the process $X$ is stationary and \textit{min--i.d.}\ (which means that $-X$ is \textit{max--i.d.}). Note that the min--i.d.\ property follows from the fact that for every $n\in\N$, the Poisson point process with constant intensity $1$ can be represented as a union of $n$ independent Poisson point processes with constant intensity $1/n$.
\begin{theorem}\label{theo:penrose_mix}
The process $X$  defined in~\eqref{def:semi_min_stab} is mixing.
\end{theorem}
\begin{proof}
In view of Theorem~\ref{theo:mix} applied to the process $-X$, we need to show that for every $a>0$, we have $\lim_{t\to\infty}\tau_a(t)=0$, where
\begin{equation}\label{eq:tau_ta}
\tau_a(t)=\log \P[X(0)\geq a, X(t)\geq a]-2\log \P[X(0)\geq a].
\end{equation}
We are going to compute the dependence coefficient $\tau_a(t)$.  Define random events $A_1$ and $A_2$ by
\begin{align*}
A_1&=\{\nexists i\in\N: \|U_i\|_2 <  a\},\\
A_2&=\{\nexists i\in\N:  \|U_i\|_2\geq a, \|U_i+W_i(t)\|_2 < a\}.
\end{align*}
Let $B(a)$ be the $d$--dimensional ball of radius $a$ around the origin, and denote by $V(a)$ its volume. By the definition of the homogeneous Poisson process, the number of $i\in\N$ such that $\|U_i\|_2<a$ is Poisson distributed with mean $V(a)$. Thus,
\begin{equation}\label{eq:p_a1}
\P[X(0)\geq a]=\P[A_1]=e^{-V(a)}.
\end{equation}
Further, the number of $i\in\N$ having the property $\|U_i\|_2\geq a$ and  $\|U_i+W_i(t)\|_2<a$ is Poisson distributed with parameter $\int_{\R^d\backslash B(a)} \P[x+W(t)\in B(a)]dx$.
Therefore,
\begin{equation}\label{eq:p_a2}
\P[A_2]=\exp\left(-\int_{\R^d\backslash B(a)} \P[x+W(t)\in B(a)]dx\right).
\end{equation}
Since the Lebesgue measure is invariant with respect to the transition semigroup of the Brownian motion, we have
$$
\int_{\R^d}\P[x+W(t)\in B(a)]dx=V(a).
$$
Inserting this into~\eqref{eq:p_a2}, we obtain
\begin{equation}\label{eq:p_a2_1}
\P[A_2]=\exp\left(-V(a)+\int_{B(a)} \P[x+W(t)\in B(a)]d x\right).
\end{equation}
By the independence property of the Poisson point process, the events $A_1$ and $A_2$ are independent. Thus,
$$
\P[X(0)\geq a, X(t)\geq a]=\P[A_1\cap A_2]=\P[A_1]\P[A_2].
$$
Hence, by~\eqref{eq:p_a1} and~\eqref{eq:p_a2_1},
\begin{equation}\label{eq:P_X0_Xt_geq_a}
\P[X(0)\geq a, X(t)\geq a]=\exp\left(-2V(a)+\int_{B(a)} \P[x+W(t)\in B(a)]d x\right).
\end{equation}
Recalling that $\tau_a(t)$ was defined in~\eqref{eq:tau_ta} and using~\eqref{eq:p_a1} and~\eqref{eq:P_X0_Xt_geq_a}, we obtain
$$
\tau_a(t)=\int_{B(a)} \P[x+W(t)\in B(a)]dx.
$$
We can estimate
$$
0\leq \tau_a(t)\leq \int_{B(a)} \P[W(t)\in B(2a)]dx=V(a) \P[W(t)\in B(2a)].
$$
Clearly, the right--hand side goes to $0$ as $t\to\infty$. Thus, $\lim_{t\to\infty}\tau_a(t)=0$. By Theorem~\ref{theo:mix}, this completes the proof.
\end{proof}
\begin{remark}
Theorem~\ref{theo:penrose_mix} can easily be generalized. For example, one may let the particles move according to a $d$--dimensional L\'evy process, or to a $d$--dimensional fractional Brownian motion.
\end{remark}


\bibliographystyle{plainnat}
\bibliography{paper16_arxiv}
\end{document}